\documentclass[12pt]{amsart}
\usepackage[alphabetic,backrefs]{amsrefs}

 \usepackage{tikz}
 
 \usetikzlibrary{decorations.pathmorphing}
 \tikzstyle{printersafe}=[decoration={snake,amplitude=0pt}]
 
 \usepackage[all]{xy}
 \usepackage{amsmath}
 \usepackage{hyperref}
 \usepackage{amsfonts,graphics,amsthm,amsfonts,amscd,latexsym}
 \usepackage{epsfig}
 \usepackage{graphicx}
 \usepackage{flafter}
 \usepackage{mathtools}
 \usepackage{comment}
 
 \hypersetup{
 	colorlinks=true,    
 	linkcolor=blue,          
 	citecolor=blue,      
 	filecolor=blue,      
 	urlcolor=blue           
 }
 
 \usepackage{tikz}
 \usetikzlibrary{graphs,positioning,arrows,shapes.misc,decorations.pathmorphing}
 
 \tikzset{
 	>=stealth,
 	every picture/.style={thick},
 	graphs/every graph/.style={empty nodes},
 }
 
 \tikzstyle{vertex}=[
 draw,
 circle,
 fill=black,
 inner sep=1pt,
 minimum width=5pt,
 ]
 \usepackage[position=top]{subfig}
 \usepackage{amssymb}
 \usepackage{color}
 \usepackage{tikz-cd}
\usepackage{geometry}
 \usepackage{enumitem}
 
 \usetikzlibrary{decorations.pathmorphing}
 \tikzstyle{printersafe}=[decoration={snake,amplitude=0pt}]
 \DeclareMathOperator{\Ima}{Im}

 \newcommand{\zz}{\mathbb{Z}}

 \newcommand{\Q}{\mathbb{Q}}

 \def\Ox #1.{\mathcal{O}_{#1}}			
 \def\pr #1.{\mathbb P^{#1}}				
 \def\af #1.{\mathbb A^{#1}}			
 \def\ses#1.#2.#3.{0\to #1\to #2\to #3 \to 0}	
 \def\xrar#1.{\xrightarrow{#1}}			
 \def\K#1.{K_{#1}}						
 \def\bA#1.{\mathbf{A}_{#1}}			
 \def\bM#1.{\mathbf{M}_{#1}}				
 \def\bL#1.{\mathbf{L}_{#1}}				
 \def\bB#1.{\mathbf{B}_{#1}}				
 \def\bK#1.{\mathbf{K}_{#1}}			
 \def\subs#1.{_{#1}}					
 \def\sups#1.{^{#1}}

 \DeclareFontFamily{U}{wncy}{}
 \DeclareFontShape{U}{wncy}{m}{n}{<->wncyr10}{}
 \DeclareSymbolFont{mcy}{U}{wncy}{m}{n}
 \DeclareMathSymbol{\Sh}{\mathord}{mcy}{"58}

 \usepackage{tikz}
 \usetikzlibrary{matrix,arrows,decorations.pathmorphing}

 \newtheorem{theorem}{Theorem}[section]
 
 \newtheoremstyle{exampstyle}
 {\topsep} 
 {\topsep} 
 {} 
 {} 
 {\bfseries} 
 {.} 
 {.5em} 
 {} 

 \newtheorem{proposition}[theorem]{Proposition}

 \theoremstyle{definition}

 \newtheorem{remark}[theorem]{Remark}
 
 \theoremstyle{remark}

 \numberwithin{equation}{section}

 \usepackage[autostyle=true]{csquotes} 
 
\title{Computing the Cassels-Tate Pairing in the case of a Richelot Isogeny}
\author{Jiali Yan }
\date{\today}
\begin{document}
\maketitle

\begin{abstract}

In this paper,  we study the Cassels-Tate pairing on Jacobians of genus two curves admitting a special type of isogenies called  Richelot isogenies. Let $\phi: J \rightarrow \widehat{J}$ be a Richelot isogeny between two Jacobians of genus two curves. We give an explicit formula as well as a practical algorithm to compute the Cassels-Tate pairing on $\text{Sel}^{\widehat{\phi}}(\widehat{J}) \times \text{Sel}^{\widehat{\phi}}(\widehat{J})$ where $\widehat{\phi}$ is the dual isogeny of $\phi$. The formula and algorithm are under the simplifying assumption that all two torsion points on $J$ are defined over $K$. We also include a worked example demonstrating we can turn the descent by Richelot isogeny into a 2-descent via computing the Cassels-Tate pairing. \\
\end{abstract}
\section{Introduction}

For any principally polarized  abelian variety $A$ defined over a number field $K$, Cassels and Tate \cite{cassels1} \cite{cassels2} and \cite{tate} constructed a pairing 
$$\Sh(A) \times \Sh(A) \rightarrow \Q/\zz,$$
that is nondegenerate after quotienting out the maximal divisible subgroup of $\Sh(A)$. This pairing is called the Cassels-Tate pairing and it natually lifts to a pairing on Selmer groups. One application of this pairing is in improving the bound on the Mordell-Weil rank $r(A)$ obtained by performing a standard descent calculation. Suppose $\Sh(A)$ is finite, then  carrying out an $n$-descent and compute the Cassels-Tate pairing on $\text{Sel}^n(A) \times \text{Sel}^n(A)$ gives the same bound as obtained from the $n^2$-descent where $\text{Sel}^{n^2}(A)$ needs to be computed as shown in \cite[Proposition 1.9.3]{thesis}.\\

There have been many results on computing the Cassels-Tate pairing in the case of elliptic curves. For example, in addition to defining the pairing, Cassels also  described a method for computing the pairing on $\text{Sel}^2(E) \times \text{Sel}^2(E)$ in \cite{cassels98} by solving conics over the field of definition of a two-torsion point.  Donnelly \cite{steve} then decribed a method that only requires solving conics over $K$ and Fisher \cite{binary quartic} used the invariant theory of binary quartics to give a new formula for the Cassels-Tate pairing on $\text{Sel}^2(E) \times \text{Sel}^2(E)$ without solving any conics. In \cite{monique} \cite{3 isogeny},  van Beek and Fisher computed the Cassels-Tate pairing on the 3-isogeny Selmer group of an elliptic curve. For $p=3$ or $5$, Fisher computed the Cassels-Tate pairing on the $p$-isogeny Selmer group of an elliptic curve in a special case  in \cite{platonic}. In \cite{3 selmer}, Fisher and Newton computed the Cassels-Tate pairing on $\text{Sel}^3(E) \times \text{Sel}^3(E)$. We are interested in the natural problem of generalizing the different algorithms for computing the Cassels-Tate pairing for elliptic curves to compute the pairing for abelian varieties of higher dimension.\\

 In this paper,  we study the Cassels-Tate pairing on Jacobians of genus two curves admitting a special type of isogeny called  a Richelot isogeny. Let $\phi: J \rightarrow \widehat{J}$ be a Richelot isogeny between Jacobians of two genus two curves. We will be working under the simplifying assumption that all two-torsion points on $J$ are defined over $K$.  Consider the following long exact sequence
 
 \begin{equation}\label{eqn: exact}
 0 \rightarrow J[\phi](\Q) \rightarrow J[2](\Q) \rightarrow  \widehat{J}[\widehat{\phi}](\Q) \rightarrow \text{Sel}^{\phi}(J) \rightarrow \text{Sel}^2(J) \xrightarrow{\alpha} \text{Sel}^{\widehat{\phi}}(\widehat{J}).
 \end{equation}
Let $ \langle \;, \; \rangle_{CT}$ denote the Cassels-Tate pairing on $\text{Sel}^{\widehat{\phi}}(\widehat{J})$. It is shown in Remark \ref{rem: richelot H^1 SES} that we can replace $\text{Sel}^{\widehat{\phi}}(\widehat{J})$ with $\ker \langle \;, \; \rangle_{CT}$ and \eqref{eqn: exact} remains exact. Thus computing the pairing $\langle \;, \; \rangle_{CT}$ potentially improves the rank bound given by carrying out a descent by Richelot isogeny.\\

 In Section 2, we give some background results needed for the later sections and we define a pairing on $\text{Sel}^{\phi}(J) \times \text{Sel}^{\phi}(J)$ following the Weil pairing definition of the Cassels-Tate pairing for the Richelot isogeny $\phi$. In Section 3, we then give an explicit formula as well as a practical algorithm to compute the Cassels-Tate pairing on $\text{Sel}^{\widehat{\phi}}(\widehat{J}) \times \text{Sel}^{\widehat{\phi}}(\widehat{J})$ where $\widehat{\phi}$ is the dual isogeny of $\phi$ and also a Richelot isogeny. In Section 4, we give some details of the explicit computation and show directly that the formula for  the Cassels-Tate pairing is always a finite product with a computable bound. In Section 5,  we include a worked example demonstrating we can turn the descent by Richelot isogeny into a 2-descent via computing the Cassels-Tate pairing. The content of this paper is based on Chapter 2 of the thesis of the author \cite{thesis}.\\

\subsection*{Acknowledgements}
I  express my sincere and deepest gratitute to my PhD supervisor, Dr. Tom Fisher, for his patient guidance and insightful comments  at every stage during my research.\\

\section{Preliminary Results}

\subsection{The set-up}\label{sec:the-set-up}
In this paper, we are working over a number field $K$. For any field $k$, we let $\bar{k}$ denote its algebraic closure and  let $\mu_n \subset \bar{k}$ denote the $n^{th}$ roots of unity in $\bar{k}$.  We let $G_k$ denote the absolute Galois group $\text{Gal}(\bar{k}/k)$.\\

Let  $\mathcal{C}$ be a general \emph{genus two curve} defined over $K$ with all Weierstrass points defined over $K$, which  is a smooth projective curve and it can be given in the following hyperelliptic form:
\begin{equation}\label{eq:eq}
 C:  y^2 = f(x) = G_1(x)G_2(x)G_3(x),
 \end{equation}
where $G_1(x) = \lambda(x -\omega_1); G_2(x) = (x -\omega_2)(x-\omega_3); G_3(x) = (x -\omega_4)(x-\omega_5) $ with $\lambda, \omega_i \in K$ and $\lambda \neq 0$.\\

We let $J$ denote the \emph{Jacobian variety} of  $\mathcal{C}$, which is an abelian variety of dimension 2 defined over $K$ that can be identified with $\text{Pic}^0(\mathcal{C})$. We denote the identity element of $J$ by $\mathcal{O}_J$. Via the natural isomorphism $\text{Pic}^2(\mathcal{C}) \rightarrow \text{Pic}^0(\mathcal{C})$ sending $[P_1+P_2] \mapsto [P_1 + P_2 - \infty^+ - \infty^-]$,  a point $P \in J$ can be identified with an unordered pair of points of $\mathcal{C}$, $\{P_1, P_2\}$. This identification is unique unless $P= \mathcal{O}_J$, in which case it can be represented by any pair of points on $\mathcal{C}$ in the form $\{(x, y), (x, -y)\}$ or $\{\infty^+, \infty^-\}$. Moreover, $J[2]=\{\mathcal{O}_J, \{(\omega_i, 0), (\omega_j, 0)\}\text{ for } i\neq j, \{(\omega_i, 0), \infty\} \}$. Let $e_2: J[2] \times J[2] \rightarrow \mu_2$ denote the Weil pairing on $J[2]$. As described in \cite[Chapter 3, Section 3]{the book}, suppose $\{P_1, P_2\} $and $\{Q_1, Q_2\}$ represent $P, Q \in J[2]$ where $P_1, P_2, Q_1, Q_2$ are Weierstrass points, then  \begin{equation}\label{eqn: e2}
e_2(P, Q)= (-1)^{|\{P_1, P_2\} \cap \{Q_1, Q_2\} |}.
\end{equation}\\

\subsection{Richelot isogenies}\label{sec:richelot-isogeny}A \emph{Richelot isogeny} is a polarized $(2, 2)$-isogeny between Jacobians of genus 2 curves. Equivalently, it is an isogeny $\phi : J \rightarrow \widehat{J}$ such that $J[\phi] \cong \zz/2\zz \times \zz/2\zz$ and $J, \widehat{J}$ are Jacobians of genus two curves. \\

 A special case of  \cite[Proposition 16.8]{abelian varieties1} and \cite[Lemma 2.4]{bruin} shows that the kernel of a Richelot isogeny is actually a maximal isotropic subgroup of $J[2]$ with respect to the Weil pairing $e_2$ on $J[2] \times J[2]$. We have the following general proposition on Richelot isogenies from \cite[Chapter 9 Section 2]{the book} and \cite[Section 3]{arbitrarily large}. In Remark \ref{rem:richelot}, we gave the extra details for the case where the hyperelliptic form of the underlying curve is degree 5.\\

\begin{proposition}\label{prop:Richelot_explicit_formulae}
	Suppose the curve $\mathcal{C}$ is of the form 
	
	$$\mathcal{C}: y^2 = f(x) = G_1(x)G_2(x)G_3(x),$$ where $G_j(x)= g_{j2}x^2 + g_{j1}x + g_{j0},$ and each $g_{ji} \in K$. Then there is a Richelot isogeny $\phi$ from $J$, the Jacobian of $\mathcal{C}$, to $\widehat{J}$, the Jacobian of the following genus two curve:
	\begin{equation}\label{eq: eq 2}
	\widehat{\mathcal{C}}: \triangle y^2 = L_1(x)L_2(x)L_3(x),
	\end{equation}
	where each $L_i(x)= G'_j(x)G_k(x)- G_j(x)G_k'(x),$ for $[i, j, k] = [1, 2, 3], [2, 3, 1], [3, 1, 2]$, and $\triangle = \det(g_{ij})$ which we assume to be non-zero. \newline
	
	In addition, the kernel of $\phi$ consists of the identity $\mathcal{O}_J$ and the 3 divisors of order 2 given by $G_i=0$. We have the similar result  for the dual isogeny $\widehat{\phi}$.\\
	
	Moreover, any genus two curve $\mathcal{C}$  that admits a Richelot isogeny with all the elements of the kernel $K$-rational is of the form $y^2 = f(x) = G_1(x)G_2(x)G_3(x)$ as above.\\
	
\end{proposition}

\begin{remark}\label{rem:richelot}
	We exclude the case $\triangle=0$ in the above proposition. In fact, by \cite[Chapter 14]{the book}, $\triangle=0$ implies that the Jacobian of $\mathcal{C}$ is the product of elliptic curves. It can be checked that the analogue of $\triangle$ for $\widehat{C}$ is $2\triangle^2$, so no need to have further condition like $\triangle \neq 0$ from $\widehat{C}$. Also, in the case where $G_i$ is linear, say $G_i=a(x-b)$, then we say $\{(b, 0), \infty\}$ is the divisor given by $G_i=0$ which gives an element in $\ker \phi$.\\
	
\end{remark}

We use the notation in Proposition \ref{prop:Richelot_explicit_formulae} and denote the nontrivial elements in the kernel of $\phi$ by $P_i$ corresponding to the divisors of order 2 given by $G_i=0$ as well as denote  the nontrivial elements in the kernel of $\widehat{\phi}$  by $P_i'$. From  \cite[Chapter 9, Section 2]{the book} and \cite[Section 3.2]{richelot}, we have the following description of the Richelot isogeny $\phi$. Associated with a Weierstrass point $P=(\omega_1, 0)$ with $G_1(\omega_1)=0$, $\phi:J \rightarrow \widehat{J}$ is given explicitly as \\
$$\{(x, y), P\} \mapsto \{(z_1, t_1),(z_2, t_2)\},$$
where $z_1, z_2$ satisfy 
$$G_2(x)L_2(z)+G_3(x)L_3(z)=0;$$
and $(z_i, t_i)$ satisfies
$$yt_i=G_2(x)L_2(z_i)(x-z_i).$$\\

Denote the set of two points on $\mathcal{C}$ given by $G_i=0$ by $S_i$ for $i=1, 2, 3$. From the explicit description above, we know that the preimages of $P_1'$ under $\phi$ are precisely $\{\{Q_1, Q_2\} \in J[2]$ such that $Q_1 \in S_2, Q_2 \in S_3\}$. Similarly we know the preimages of $P_2'$ and $P_3'$.\\



\subsection{The Weil pairing for the Richelot isogeny}\label{sec:weil-pairing-for-richelot-isogeny}

Let $J$ and $\widehat{J}$ be Jacobian varieties of genus two curves defined over $K$. Assume there is  a Richelot isogeny $\phi: J \rightarrow \widehat{J}$ with $\widehat{\phi}$ being its dual, i.e. $\phi \circ \widehat{\phi} = [2]$. Then we have the Weil pairing 
$$e_{\phi} : J[\phi] \times \widehat{J}[\widehat{\phi}] \rightarrow \bar{K}^*,$$
where  $e_{\phi}(P, Q)= e_{2, J}(P, Q')$ for any $Q' \in J[2]$ such that $\phi(Q')=Q$. The image of $e_{\phi}$ is  in fact $\mu_2(\bar{K}^*) \subset \bar{K}^*$. Recall $J[ \phi]$ is isotropic with respect to $e_{2, J}$ as discussed in \ref{sec:richelot-isogeny}. This implies that $e_{2, J}(P, Q')=e_{2, J}(P, Q'')$ if $\phi(Q')=\phi(Q'')$ and hence $e_{\phi}$ is well-defined. By \eqref{eqn: e2} and the end of Section \ref{sec:richelot-isogeny}, $e_{\phi}(P_i, P_i')=1$ for any $i=1, 2, 3$ and $e_{\phi}(P_i, P_j')=-1$ for any $i \neq j$. \\

\subsection{Definition of the Cassels-Tate pairng in the case of a Richelot isogeny}\label{sec:definition-of-the-cassels-tate-pairng-in-the-case-of-a-richelot-isogeny}

In this section, we give the definition of the Cassels-Tate pairing in the case of a Richelot isogeny. There are four equivalent definitions of the Cassels-Tate pairing stated and proved in \cite{poonen stoll}. The compatibility of the definition below with the Weil Pairing definition of the Cassels-Tate pairing on $\text{Sel}^2(J) \times \text{Sel}^2(J)$ is shown in  \cite[Proposition 2.1.6]{thesis}.\\

Let $(J, \lambda_1)$ and $(\widehat{J}, \lambda_2)$ be Jacobian varieties of genus two curves defined over a number field $K$ such that there exists a Richelot isogeny $\phi: J \rightarrow \widehat{J}$ with $\widehat{\phi}: \widehat{J}\rightarrow J$ being its dual isogeny. It is shown in \cite[Lemma 2.1.4]{thesis} that for any $b \in \text{Sel}^{\phi}(J)$, there exists $b_1 \in H^1(G_K, \widehat{J}[2])$ mapping to $b$ under the map induced by $\widehat{J}[2] \xrightarrow{\widehat{\phi}} J[\phi]$.\\

\textbf{The definition of the pairing}\\

Let $a, a' \in \text{Sel}^{\phi}(J)$. Suppose $a_1 \in H^1(G_K, \widehat{J}[2])$ maps to $a \in \text{Sel}^{\phi}(J) \subset H^1(G_K, J[\phi])$ under the map induced by $\widehat{J}[2] \xrightarrow{\widehat{\phi}} J[\phi]$.\\

Consider the commutative diagram below.\\
\[
\begin{tikzcd}
\widehat{J}(K_v) \arrow[r, "\widehat{\phi}"'] \arrow[d, "="']&J(K_v) \arrow[r, "\delta_{\widehat{\phi}}"'] \arrow[d, "\phi"']&H^1(G_{K_v}, \widehat{J}[\widehat{\phi}]) \arrow[d, "\iota"', "\rho_v \mapsto \delta_{2}(P_v) -a_{1,v}"]\\
\widehat{J}(K_v) \arrow[r, "2"'] \arrow[d, "\widehat{\phi}"']&\widehat{J}(K_v) \arrow[r, "\delta_{2}"'] \arrow[d, "="']&H^1(G_{K_v}, \widehat{J}[2]) \arrow[d, "\widehat{\phi}"', "\delta_{2}(P_v) \mapsto a_v \;\; a_{1,v} \mapsto a_v"]\\
J(K_v) \arrow[r, "\phi"']&\widehat{J}(K_v) \arrow[r, "\delta_{\phi}"', "P_v \mapsto a_v"]&H^1(G_{K_v}, J[\phi])\\
\end{tikzcd} \]

Let $v$ be a place of $K$. Let $P_v \in \widehat{J}(K_v)$ be a lift of $a_v \in H^1(G_{K_v}, J[\phi])$. Then $\delta_{2}(P_v)$ and $a_{1, v}$ in $H^1(G_{K_v}, \widehat{J}[2])$ both map to $a_v$. Hence, we may choose  $\rho_v \in H^1(G_{K_v}, \widehat{J}[\widehat{\phi}])$ a lift of  $\delta_{2}(P_v)-a_{1, v}$ and define $\eta_v = \rho_v \cup_{\widehat{\phi}, v} a_v' \in H^2(G_{K_v}, \bar{K_v}^*)$. Here $\cup_{\widehat{\phi}, v}$ denotes  the cup product $H^1(G_{K_v}, \widehat{J}[\widehat{\phi}]) \times H^1(G_{K_v}, J[\phi]) \rightarrow H^2(G_{K_v}, \bar{K}^*)$ associated to $e_{\widehat{\phi}}$. The Cassels-Tate pairing is defined by\\
$$\langle a, a'\rangle_{CT} := \sum_{v}\text{inv}_v(\eta_v).$$

We sometimes refer to $\text{inv}_v(\eta_v)$  above as the local Cassels-Tate pairing between $a, a' \in \text{Sel}^{\phi}(J)$ for a place $v$ of $K$, noting that this depends on the choice of the global lift $a_1$.\\

\begin{remark}
	In \cite[Proposition 1.8.4]{thesis}, the Weil pairing definition of the Cassels-Tate pairing is proved to be independent of the choices made in the definition. Since the above pairing is compatible with the Weil pairing definition as in  \cite[Proposition 2.1.6]{thesis}, we know it is also independent of the choices we make.\\

\end{remark}

\section{Computation of the Cassels-Tate Pairing}

 Recall that we are working with a genus curve $\mathcal{C}$ in the form \eqref{eq:eq} and we fix a choice of Richelot isogeny  $\phi: J \rightarrow \widehat{J}$ where $J$ is the Jacobian of $\mathcal{C}$ and $\widehat{J}$ is the Jacobian of the genus two curve defined by \eqref{eq: eq 2}. We write $\widehat{\phi}$ for the dual of $\phi$. This implies that all points in $J[2]$ are defined over $K$ and all points in $\widehat{J}[\widehat{\phi}]$ are defined over $K$ by Proposition \ref{prop:Richelot_explicit_formulae}. Recall, we denote the nontrivial elements in $J[\phi]$ by $P_1, P_2, P_3$ where $P_i$ corresponds to the divisor given by $G_i=0$ and the nontrivial elements in $\widehat{J}[\widehat{\phi}]$ by $P_1', P_2', P_3'$ where $P_i'$ corresponds to the divisor given by $L_i=0$ as in the same Proposition. In this section, we will give a practical formula for the explicit computation for the Cassels-Tate pairing in the case of Richelot isogenies.\\

\subsection{Explicit embeddings of $\mathbf{H^1(G_K, J[\phi])}$ and $\mathbf{H^1(G_K, J[2])}$}\label{sec:explicit-embeddings}
In order to give the  formula for the Cassels-Tate pairing,  we first describe some well-known embeddings that are useful for the explicit computation.\\

Recall  all points in $J[2]$ and $\widehat{J}[\widehat{\phi}]$ are defined over $K$. From the exact sequence
$$0 \rightarrow J[\phi] \xrightarrow{w_{\phi}} (\mu_2)^3 \xrightarrow{N} \mu_2 \rightarrow 0 ,$$
where ${w_{\phi}}: P \mapsto (e_{\phi}(P, P_1'), e_{\phi}(P, P_2'), e_{\phi}(P, P_3'))$  and $N: (a, b, c) \mapsto abc$, we  get

$$H^1(G_K, J[\phi]) \xrightarrow{inj} H^1(G_K, (\mu_2)^3) \cong (K^*/(K^*)^2)^3 \xrightarrow{N_*} H^1(G_K, \mu_2) \cong K^*/(K^*)^2,$$
where $\cong$ denotes the Kummer isomorphism derived from Hilbert's Theorem 90 and $N_*$ is induced by $N$. The induced map $H^1(G_K, J[\phi]) \rightarrow H^1(G_K, (\mu_2)^3)$ is injective as the map $(\mu_2)^3 \xrightarrow{N} \mu_2$ is surjective. Furthermore, the image of this  injection contains precisely all the elements with norm a square by the exactness of the sequence above, i.e. $H^1(G_K, J[\phi])\cong \ker ((K^*/(K^*)^2)^3 \xrightarrow{N_*} K^*/(K^*)^2)$. We have a similar embedding for $H^1(G_K,\widehat{J}[\widehat{\phi}])$. \\

Also, from the exact sequence
$$0 \rightarrow J[2] \xrightarrow{w_2} (\mu_2)^5 \xrightarrow{N} \mu_2 \rightarrow 0, $$
where  ${w_2}:P \mapsto (e_2(P, \{(\omega_1, 0), \infty\}), ..., e_2(P, \{(\omega_5, 0), \infty\}))$ and $N: (a, b, c, d, e) \mapsto abcde$, we get

$$H^1(G_K, J[2]) \xrightarrow{inj} H^1(G_K, (\mu_2)^5) \cong (K^*/(K^*)^2)^5 \xrightarrow{N_*} H^1(G_K, \mu_2) \cong K^*/(K^*)^2,$$
where $\cong$ denotes the Kummer isomorphism derived from Hilbert's Theorem 90 and $N_*$ is induced by $N$. Again the induced map $H^1(G_K, J[2]) \rightarrow H^1(G_K, (\mu_2)^5)$ is injective as the map $(\mu_2)^5 \xrightarrow{N} \mu_2$ is surjective. Furthermore, the image of this injection also contains precisely all the elements with norm a square from the exact sequence above. In particular, we have $$H^1(G_K, J[2]) \cong (K^*/(K^*)^2)^4.$$\\

\subsection{Explicit Formula}\label{sec:explicit-formula}
Using the embeddings described in Section \ref{sec:explicit-embeddings}, we can now state and prove the explicit formula for the Cassels-Tate pairing in the case of a Richelot isogeny.\\

\begin{proposition}\label{prop:explicit weil pairing cup }
	
	Under the embeddings of $H^1(G_K, J[\phi])$ and $H^1(G_K, \widehat{J}[\widehat{\phi}])$ in $(K^*/(K^*)^2)^3$ as described in Section \ref{sec:explicit-embeddings}, we get that the cup product $\cup_{\phi}$ induced by $e_{\phi}$ is
	$$H^1(G_K, J[\phi]) \times H^1(G_K, \widehat{J}[\widehat{\phi}]) \rightarrow \mathrm{Br}(K)[2]$$
	$$((a_1, b_1, c_1), (a_2, b_2, c_2)) \mapsto(a_1, a_2)+(b_1, b_2)+(c_1, c_2),$$
	where $(\;,\;)$ represents the quaternion algebra and also its equivalence class in $\mathrm{Br}(K)[2]$.\\

\end{proposition}
\begin{proof}
	Recall that the embedding $J[\phi]\rightarrow (\mu_2)^3$ is given by sending $P\in J[\phi]$ to $$ (e_\phi(P,P_1'),e_\phi(P,P_2'),e_\phi(P,P_3'))$$ and the embedding $\widehat{J}[\widehat{\phi}]\rightarrow (\mu_2)^3$ is given by sending $Q \in \widehat{J}[\widehat{\phi}] $ to $$(e_\phi(P_1,Q),e_\phi(P_2,Q),e_\phi(P_3,Q)).$$
	
	It can be checked, via the end of discussion of Section \ref{sec:richelot-isogeny}, that we have the following commutative diagram:
	\[
	\begin{tikzcd}
	J[\phi] \times \widehat{J}[\widehat{\phi}] \arrow[r,"inj"] \arrow[d, "e_{\phi}"'] & (\mu_2)^3 \times (\mu_2)^3 \arrow[d, "f"]\\
	\mu_2 \arrow[r, "="]& \mu_2,
	\end{tikzcd}
	\]
	where $f$ sends $((-1)^a, (-1)^b, (-1)^c), ((-1)^{a'}, (-1)^{b'}, (-1)^{c'})$ to $  (-1)^{aa'+bb'+cc'} $ with $a, b, c \in \{0, 1\}$.\\
	
	Consider the natural pairing $\phi: \mu_2 \times \mu_2\rightarrow \mu_2$ sending $((-1)^a,(-1)^b)$ to $(-1)^{ab}$.
	This gives a cup product pairing 
	$$
	\begin{array}{ccc}
	H^1(G_K, \mu_2) \times H^1(G_K, \mu_2) & \longrightarrow & H^2(G_K, \mu_2) \cong \text{Br}(K)[2]\\
	([\sigma \mapsto a_{\sigma}], [\tau \mapsto b_{\tau}]) & \longmapsto& [(\sigma, \tau) \mapsto \phi(a_{\sigma}, b_{\tau})]
	\end{array}.$$
	By  Hilbert's Theorem 90, we can identify $H^1(G_K, \mu_2)$ with $K^*/(K^*)^2$. Under this identification, the image of $(a, b) \in K^*/(K^*)^2 \times K^*/(K^*)^2$ is precisely the equivalence class of the quaternion algebra $(a, b)$ by  \cite[Chapter XIV, Section 2, Proposition 5]{local fields} and \cite[Corollary 2.5.5(1), Proposition 4.7.3]{QA}.\\
	
	Therefore, we get that the induced cup product is
	$$H^1(K, J[\phi]) \times H^1(K, \widehat{J}(\widehat{\phi})) \rightarrow \text{Br}(K)[2]$$
	$$((a_1, b_1, c_1), (a_2, b_2, c_2)) \mapsto(a_1, a_2)+(b_1, b_2)+(c_1, c_2).$$

\end{proof}

\begin{proposition}\label{prop: explicit Richelot map 1}
	
	Under the embeddings of $H^1(G_K, J[\phi])$ and $ H^1(G_K, J[2])$ in $(K^*/(K^*)^2)^3$ and $(K^*/(K^*)^2)^5$ respectively as described in Section \ref{sec:explicit-embeddings}, the map $\Psi: H^1(G_K, J[\phi]) \rightarrow H^1(G_K, J[2])$ induced from  the inclusion $J[\phi] \rightarrow  J[2]$ is given by  
	
	
	$$(a, b, c) \mapsto (1, c, c, b, b).$$
	\\
	
\end{proposition}

\begin{proof}\label{proof: explicit richelot map 1}
	
	Recall the embedding of $H^1(G_K, J[2])$ in $(K^*/(K^*)^2)^5 $, and  the embedding of $H^1(G_K, J[\phi])$ in $(K^*/(K^*)^2)^3$ are induced from the short exact sequences in the following commutative diagram:
	\[\begin{tikzcd}
	0 \arrow[r]&J[\phi] \arrow[r, "{w_{\phi}}"] \arrow[d, "inc"]& (\mu_2)^3 \arrow[r, "N"] \arrow[d, "\psi"]& \mu_2 \arrow[r] \arrow[d, "="]& 0\\
	0 \arrow[r]&J[2] \arrow[r, "{w_2}"]& (\mu_2)^5 \arrow[r, "N"]& \mu_2 \arrow[r]& 0.\\
	\end{tikzcd}\]
	
	Suppose $P \in J[\phi]$ maps to $(\alpha, \beta, \gamma)$ via $w_{\phi}$. Then $e_{\phi}(P, P_1')=\alpha, e_{\phi}(P, P_2')=\beta, e_{\phi}(P, P_3')=\gamma.$ By definition, $e_{\phi}(P, \phi(Q))=e_2(P, Q)$ for any $Q \in J[2]$. From the explicit description of $\phi$ in Section \ref{sec:richelot-isogeny}, we know $ \alpha= e_{2}(P,  \{(\omega_2, 0), (\omega_4, 0)\})$, $\beta= e_{2}(P,  \{(\omega_1, 0), (\omega_5, 0) \})$ and $\gamma = e_{2}(P,\{\infty, (\omega_3, 0) \})$. Recall that $J[\phi]$ is isotropic with respect to $e_2$. This implies that $w_{2}(P) = (1, \gamma, \gamma, \beta, \beta)$. Therefore,  we define $\psi(\alpha, \beta, \gamma) = (1, \gamma, \gamma, \beta, \beta)$, which makes the above diagram commute.\\
	
	
	
	Now consider $\Psi: H^1(G_K, J[\phi])  \rightarrow H^1(G_K, J[2])$ which, via the embedding in Section \ref{sec:explicit-embeddings}, is the map $H^1(G_K, (\mu_2)^3) \rightarrow H^1(G_K, (\mu_2)^5)$ induced by  $\psi$.  It can be checked that $\Psi(a, b, c) = (1, c, c, b, b).$\\
	
\end{proof}
\begin{proposition}\label{prop: explicit Richelot map 2}
	
	Under the embeddings of $H^1(G_K, \widehat{J}[\widehat{\phi}])$ and $H^1(G_K, J[2])$ in $(K^*/(K^*)^2)^3$ and  $(K^*/(K^*)^2)^5$ respectively as described in Section \ref{sec:explicit-embeddings}, the map $\Psi: H^1(G_K, J[2]) \rightarrow H^1(G_K, \widehat{J}[\widehat{\phi}])$ induced from  $J[2] \xrightarrow{\phi} \widehat{J}[\widehat{\phi}]$ is given by 
	$$(a_1, a_2, a_3, a_4, a_5) \mapsto (a_1, a_2a_3, a_4a_5).$$\\
	
\end{proposition}
\begin{proof}\label{proof: explicit richelot map 2}
	
	Consider the following commutative diagram whose rows are exact sequences  
	
	\[\begin{tikzcd}
	0 \arrow[r]&J[2] \arrow[r, "{w_2}"] \arrow[d, "\phi"]& (\mu_2)^5 \arrow[r, "N"] \arrow[d, "\psi"]& \mu_2 \arrow[r] \arrow[d, "="]& 0\\
	0 \arrow[r]&\widehat{J}[\widehat{\phi}] \arrow[r, "{w_{\widehat{\phi}}}"] & (\mu_2)^3 \arrow[r, "N"] & \mu_2 \arrow[r] & 0.
	\end{tikzcd}\]\\
	
	Suppose $P \in J[2]$  maps to $(\alpha_1, \alpha_2, \alpha_3, \alpha_4, \alpha_5)$ via $w_2$. Then  $\alpha_i=e_2(P, \{(\omega_i, 0), \infty\})$. Recall that $e_{\widehat{\phi}}(\phi(P), P_i)=e_2(P, P_i)$ by the discussion at the end of Section \ref{sec:weil-pairing-for-richelot-isogeny}. This implies that $\phi(P)$ maps to $(\alpha_1, \alpha_2\alpha_3, \alpha_4\alpha_5)$ via $w_{\widehat{\phi}}$. Therefore, we can verify that the induced map $\Psi: H^1(G_K, J[2]) \rightarrow H^1(G_K, \widehat{J}[\widehat{\phi}])$ under the embeddings in Section \ref{sec:explicit-embeddings} is  given by $$(a_1, a_2, a_3, a_4, a_5) \mapsto (a_1, a_2a_3, a_4 a_5).$$\\
	
	\begin{remark}\label{rem: richelot H^1 SES}
		We observe that, under the assumption of this section, we have the following short exact sequence:
		$$0 \rightarrow H^1(G_K, J[\phi]) \rightarrow H^1(G_K, J[2]) \rightarrow H^1(G_K, \widehat{J}[\widehat{\phi}]) \rightarrow 0.$$
		The injectivity of the map $H^1(G_K, J[\phi]) \rightarrow H^1(G_K, J[2])$ is due to the surjectivity of $J(K)[2] \xrightarrow{\phi} \widehat{J}(K)[\widehat{\phi}]$. Observe that the element in $H^1(G_K, \widehat{J}[\widehat{\phi}])$ represented by $(a, b, c)$ has a preimage in $H^1(G_K, J[2])$ represented by $(a, 1, b, 1, c)$ by Proposition \ref{prop: explicit Richelot map 2}. This implies that $H^1(G_K, J[2]) \rightarrow H^1(G_K, \widehat{J}[\widehat{\phi}])$ is surjective.\\
		
	\end{remark}
	
	\begin{remark}
		Let $v$ be a place of $K$. We also have the explicit embeddings of $H^1(G_K, J[\phi])$ and $H^1(G_K, J[2])$ described in Section \ref{sec:explicit-embeddings} as well as the explicit maps given in this section if we replace $K$ with $K_v$ or $K_v^{nr}$.\\
		
	\end{remark}
\end{proof}

Using the above three propositions, we now have the explicit formula for the Cassels-Tate pairing in the case of a Richelot isogeny.\\

\begin{theorem}\label{thm: CTP}
	Let $J$ be the Jacobian variety of a genus two curve defined over a number field $K$. Suppose all points in $J[2]$ are defined over $K$ and there exists a Richelot isogeny $\phi: J \rightarrow \widehat{J}$ where $\widehat{J}$ is the Jacobian variety of another genus two curve. Let $\widehat{\phi}$ be the dual isogeny of $\phi$. Consider $a, a' \in \text{Sel}^{\widehat{\phi}}(\widehat{J})$. Suppose $(\alpha_1', \alpha_2', \alpha_3') \in (K^*/(K^*)^2)^3$ represents $a'$. For any place $v$, we let $P_v \in J(K_v)$  denote a lift of $a_v \in H^1(G_{K_v}, \widehat{J}[\widehat{\phi}])$ and suppose $\delta_2(P_v) \in H^1(G_{K_v}, J[2])$ is represented by $(x_{1, v}, x_{2, v},x_{3, v},x_{4, v},x_{5, v}) \in (K^*/(K^*)^2)^5$. Then we have 
	$$\langle a, a' \rangle_{CT}=\prod_v(x_{2, v}x_{4, v}, \alpha_1')_v(x_{4, v}, \alpha_2')_v(x_{2, v},\alpha_3')_v,$$
	where $(\;, \;)_v$ represents the Hilbert symbol.\\

	\end{theorem}
\begin{proof}
	Suppose $a$ is represented by $(\alpha_1, \alpha_2, \alpha_3)\in (K^*/(K^*)^2)^3$. Then it has a preimage $a_1 \in H^1(G_K, J[2])$ represented by $(\alpha_1, 1, \alpha_2, 1, \alpha_3)$ by Proposition \ref{prop: explicit Richelot map 2}. So following the definition of $\langle a, a'\rangle_{CT}$, we need to compute $\rho_v \cup_{\phi, v} \eta \in H^2(G_{K_v}, \bar{K_v}^*)$ where $\rho_v \in H^1(G_{K_v}, J[\phi])$ is a lift of $\delta_2(P_v)-a_{1, v}$ and  $\cup_{\phi, v}$ is the cup product induced by $e_{\phi}$. Since $\delta_2(P_v)-a_{1, v}$ is represented by $(x_{1, v}/\alpha_1, x_{2, v}, x_{3, v}/\alpha_2, x_{4, v}, x_{5, v}/\alpha_3)$, by Proposition \ref{prop: explicit Richelot map 1}, $\rho_v$ is represented by $(x_{2, v}x_{4, v}, x_{4, v}, x_{2, v})$. Hence, by Proposition \ref{prop:explicit weil pairing cup }, we know $\langle \epsilon, \eta\rangle_{CT}=  \sum_{v}\text{inv}_v((x_{2, v}x_{4, v}, \alpha_1')+(x_{4, v}, \alpha_2')+(x_{2, v},\alpha_3'))=\prod_v (x_{2, v}x_{4, v}, \alpha_1')_v(x_{4, v}, \alpha_2')_v(x_{2, v},\alpha_3')_v.$\\

	\end{proof}

\section{Computational details}
In this section, we will describe some further details for the explicit computation for the Cassels-Tate pairing using the formula in Theorem \ref{thm: CTP}.

\subsection{Embedding of $\mathbf{\widehat{J}(K)/\phi(J(K))}$ and $\mathbf{J(K)/2J(K)}$}

General results show that we have the injection, which is the composition of the connecting map $\delta_{\phi}: \widehat{J}(K)/\phi(J(K)) \rightarrow H^1(G_K, J[\phi])$ and the embedding described above $H^1(G_K, J[\phi]) \rightarrow (K^*/(K^*)^2)^3$. This is discussed  in \cite[Section 3]{arbitrarily large} \cite[Chapter 10 Section 2]{the book}. More explicitly, we have

$$\begin{array}{cccc}
\mu^{\phi}: & \widehat{J}(K)/\phi(J(K))&\longrightarrow&K^*/(K^*)^2\times K^*/(K^*)^2 \times K^*/(K^*)^2  \\
&\{(x_1, y_1), (x_2, y_2)\}&\longmapsto&(L_1(x_1)L_1(x_2), L_2(x_1)L_2(x_2), L_3(x_1)L_3(x_2))
\end{array}.$$
Similarly we have the injection:
\begin{equation}\label{eqn: mu}
\begin{array}{cccc}
\mu^{\widehat{\phi}}: & J(K)/\widehat{\phi}(\widehat{J}(K))&\longrightarrow&K^*/(K^*)^2\times K^*/(K^*)^2\times K^*/(K^*)^2  \\
&\{(x_1, y_1), (x_2, y_2)\}&\longmapsto&(G_1(x_1)G_1(x_2), G_2(x_1)G_2(x_2), G_3(x_1)G_3(x_2))
\end{array}.
\end{equation}
Note the following special cases. When $x_j$ is a root of $G_i$, then $G_i(x_j)$ should be taken to be $\prod_{l \in \{1, 2, 3\}\setminus \{i\}}G_l(x_j)$. We have a similar solution when $x_j$ is a root of $L_i$, which is replacing $L_i(x_j)$  with $\triangle \prod_{l \in \{1, 2, 3\}\setminus \{i\}}L_l(x_j)$.  When $(x_j, y_j)=\infty$, then $G_i(x_j)$ is taken to be 1. In the case where  one of $L_i$ is linear and $(x_j, y_j)=\infty$, then $L_i(x_j)$ is taken to be 1.\\

On the other hand, we have a standard injection, which is the composition of the connecting map $\delta_2: J(K)/2J(K) \rightarrow H^1(G_K, J[2])$ and the embedding described above $H^1(G_K, J[2]) \rightarrow (K^*/(K^*)^2)^5$. This can also be found in \cite[Section 3]{arbitrarily large} \cite[Chapter 10 Section 2]{the book}.

$$\begin{array}{cccc}
\mu: & J(K)/2J(K)&\longrightarrow&(K^*/(K^*)^2)^5  \\
&\{(x_1, y_1), (x_2, y_2)\}&\longmapsto& ((x_1-\omega_1)(x_2-\omega_1), ..., (x_1-\omega_5)(x_2-\omega_5))\\
\end{array}.$$
Note the following special cases. When $(x_j, y_j) = (\omega_i, 0)$, then $x_j - \omega_i$ should be taken to be $\lambda \prod_{l \in {1, 2, 3, 4, 5} \setminus \{i\}}(\omega_i - \omega_l).$ When $(x_j, y_j)=\infty$, then $x_j-\omega_i$ is taken to be $\lambda$.\\

Observe that the image of the maps $\mu^{\phi}$ and $\mu^{\widehat{\phi}}$ are both contained in the kernel of $(K^*/(K^*)^2)^3 \xrightarrow{N} K^*/(K^*)^2$. Similarly, the image of $\mu$ is contained in the kernel of $(K^*/(K^*)^2)^5 \xrightarrow{N} K^*/(K^*)^2$. \\

\subsection{Bounding the set of bad primes}
The formula for the local Cassels-Tate pairing in Theorem \ref{thm: CTP} is trivial outside the finite set of places  $S=  \{\text{places of bad reduction for } \mathcal{C}\} \cup \{\text{places dividing } 2\} \cup \{\text{infinite places}\}.$ This is explained as follows.\\

By \cite[Chapter I, Section 6]{arithmetic duality} \cite[Section 3]{selmer}, we have $\text{Sel}^{\phi}(J) \subset H^1(G_K, J[\phi]; S)=\ker \left(H^1(G_K, J[\phi]) \rightarrow \prod_{v \notin S} H^1(G_{K_v^{nr}}, J[\phi])\right)$. Similarly, $\text{Sel}^{\widehat{\phi}}(\widehat{J}) \subset H^1(G_K, \widehat{J}[\widehat{\phi}]; S)$ and $\text{Sel}^2(J) \subset H^1(G_K, J[2]; S)$. It can be shown that $\ker\big(K^*/(K^*)^2 \rightarrow \prod_{v \notin S}K_v^{nr *}/(K_v^{nr *})^2 \big)=K(S, 2)$, where $K(S, 2)$ is defined to be $\{x \in K^*/(K^*)^2: \text{ord}_v(x) \text{ is even for all } v\notin S \}$.  So, we have  $\alpha_i, \alpha'_i \in K(S, 2)$ for all $i$. Suppose $v \notin S$. We know there exists a representation of the image of $a_{1, v}$ in $(K_v^*/(K_v^*)^2)^5$ such that all its coordinates have valuation 0. Since $J(K_v^{nr}) \xrightarrow{2} J(K_v^{nr})$ is surjective by \cite[Lemma 3.4]{surj}, the map $H^0(G_{K_v^{nr}}, J) \rightarrow H^1(G_{K_v^{nr}}, J[2])$
is the zero map and hence the image of $P_v$ is trivial in $H^1(G_{K_v^{nr}}, J[2]).$ This implies that $\delta_2(P_v) \in H^1(G_{K_v}, J[2]) \subset (K_v^*/(K_v^*)^2)^5$ has a representation such that all its coordinates have valuation 0. This implies that $\delta_2(P_v)- a_{1, v}\in H^1(G_{K_v}, J[2]) \subset (K_v^*/(K_v^*)^2)^5$ has a representation such that all its coordinates have valuation 0. Then, by the formula in Proposition \ref{prop: explicit Richelot map 1}, $\rho_v \in H^1(G_{K_v}, \widehat{J}[\widehat{\phi}]) \subset (K_v^*/(K_v^*)^2)^3$ also has a representation such that all its coordinates have valuation 0. From the first part of the theorem, we know computing $\langle a, a'\rangle_{CT}$ requires computing the Hilbert symbol. It can be checked that the Hilbert symbol between $x$ and $y$ is trivial when  the valuations of $x$, $y$ are both 0 and the local field has odd residue characteristic. Hence, the local Cassels-Tate pairing is trivial for all but finitely many places contained in the set $S$.\\

\section{Worked Example}

We  explicitly compute the Cassels-Tate pairing in an example where this  improves the rank bound obtained via descent by Richelot isogeny. We will be using the same notation as in Section \ref{sec:definition-of-the-cassels-tate-pairng-in-the-case-of-a-richelot-isogeny} to compute $\langle \;,\;\rangle_{CT}$ on $\text{Sel}^{\widehat{\phi}}(\widehat{J}) \times \text{Sel}^{\widehat{\phi}}(\widehat{J})$. Our base field $K$ is the field of the rationals, $\Q$. \\

Let us consider the following genus two curve which is obtained by taking $k=113$ in \cite[Theorem 1]{arbitrarily large}
$$\mathcal{C}: y^2 = (x + 2 \cdot 113)x(x-6 \cdot113)(x +113)(x - 7 \cdot 113),$$ 
with $G_1 = (x + 2 \cdot 113),
G_2 = x(x-6 \cdot113),
G_3 = (x +113)(x - 7 \cdot 113) $ and

$$\triangle = \begin{bmatrix}
2 \cdot 113&1&0\\
0& -6 \cdot 113 &1\\
-7 \cdot 113^2& -6 \cdot 113 &1
\end{bmatrix}\\
=-7 \cdot 113^2,$$
\begin{align*}
&L_1 = G_2'G_3 - G_3'G_2 =-14 \cdot 113^2(x-3 \cdot 113),\\
&L_2 = G_3'G_1 - G_1'G_3 = (x + 5 \cdot 113)(x - 113),\\
&L_3 = G_1'G_2 - G_2'G_1 = -(x + 6 \cdot 113)(x - 2 \cdot 113).\\
\end{align*}

So we have a Richelot isogeny $\phi$ from $J$, the Jacobian variety of $\mathcal{C}$, to $\widehat{J}$, the Jacobian variety of the following curve.
$$\mathcal{\widehat{C}}: y^2 = -2(x -3 \cdot 113)(x + 5 \cdot 113)(x - 113)(x + 6 \cdot 113)(x - 2 \cdot 113)  $$

It can be shown, using MAGMA \cite{magma}, that: 

\begin{align}\label{eqn: sel}
&\text{Sel}^{\widehat{\phi}}(\widehat{J}) = \langle(2 \cdot 113, -14 \cdot 113, -7), (113, 7, 7 \cdot 113), (113, 113, 1), (2, 2, 1), (1, 7, 7)\rangle \subset (\Q^*/(\Q^*)^2)^3,\\
&\text{Sel}^{\phi}(J) = \langle(113, -7 \cdot 113, -7), (2 \cdot 113, 7, 14 \cdot 113), (113, 1, 113)\rangle \subset (\Q^*/(\Q^*)^2)^3.
\end{align}

Now we will compute the Cassels-Tate pairing matrix on $\text{Sel}^{\widehat{\phi}}(\widehat{J}) \times \text{Sel}^{\widehat{\phi}}(\widehat{J}).$ Since $(2 \cdot 113, -14 \cdot 113, -7), (113, 7, 7 \cdot 113)$ are images of elements $\{(0, 0), (-2 \cdot 113, 0)\}$ and $\{(-2\cdot 113, 0), (-113, 0)\}$ in $J(\Q)/\widehat{\phi}(\widehat{J}(\Q))$ via $\mu^{\widehat{\phi}}$ in \eqref{eqn: mu}, they are in the kernel of the Cassels-Tate pairing. So it is sufficient to look at the pairing on $\langle (113, 113, 1), (2, 2, 1), (1, 7, 7)\rangle\times \langle (113, 113, 1), (2, 2, 1), (1, 7, 7)\rangle$.\\

Since the primes of bad reduction are $\{2, 3, 7, 113\}$, we know these are the only primes for which we need to consider by Theorem \ref{thm: CTP}. The tables below give details of the local computations at these primes.\\

Let $a = (113, 113, 1) \in \text{Sel}^{\widehat{\phi}}(\widehat{J})$. By the formula given in Proposition \ref{prop: explicit Richelot map 2}, it has a lift $a_1=(113,1, 113, 1, 1) \in H^1(G_K, J[2]).$ Then for the local calculation, we have the following table:\\

\begin{center}
	\begin{tabular}{||c|ccccc}
		
		place $v$ & $\infty$& 2&3 &7&	113\\
		\hline
		$P_v$ &id&id&$\{(0, 0), (-113, 0)\}$&id&$\{(0, 0), (-2 \cdot 113, 0)\}$\\
		$\delta_2(P_v)$&id&id&$(-1, 3, -3, -1, -1)$&id&(113, 3 $\cdot$113, 3, 1, 1)\\
		
		$a_{1, v}$&id&id&$(-1, 1, -1, 1, 1)$&id&(113, 1, 113, 1, 1)\\
		
		$\delta_2(P_v)-a_{1, v}$&id&id&$(1, 3, 3, -1, -1)$&id&(1, 3 $\cdot$113, 3 $\cdot$113, 1, 1)\\
		
		$\rho_v$&id&id&$(-3, -1, 3)$&id&(3$\cdot$113, 1, 3$\cdot$113)\\
		
	\end{tabular}
\end{center}
$\;$\\

Now let $a = (2,2, 1) \in \text{Sel}^{\widehat{\phi}}(\widehat{J})$. By the formula given in Proposition \ref{prop: explicit Richelot map 2}, it has a lift $a_1= (2,1, 2, 1, 1) \in H^1(G_K, J[2]). $
Then for the local calculation, we have the following table:\\

\begin{center}
	\begin{tabular}{||c|ccccc}
		
		place $v$ & $\infty$& 2&3 &7&	113\\
		\hline
		$P_v$ &id&$\{(0, 0), (-2 \cdot 113, 0)\}$&$\{(0, 0), (-113, 0)\}$&id&id\\
		$\delta_2(P_v)$&id&$(2, 6, 3, -1, -1)$&$(-1, 3, -3, -1, -1)$&id&id\\
		
		$a_{1, v}$&id&(2, 1, 2, 1, 1)&$(-1, 1, -1, 1, 1)$&id&id\\
		
		$\delta_2(P_v)-a_{1, v}$&id&$(1,6, 6, -1, -1)$&$(1, 3, 3, -1, -1)$&id&id\\
		
		$\rho_v$&id&$(-6, -1, 6)$&$(-3, -1, 3)$&id&id\\
		
	\end{tabular}
\end{center}
$\;$\\

Lastly  let $a = (1, 7, 7) \in \text{Sel}^{\widehat{\phi}}(\widehat{J})$. By the formula given in Proposition \ref{prop: explicit Richelot map 2}, it has a lift $a_1= (1, 1, 7, 1, 7) \in H^1(G_K, J[2]). $
Then for the local calculation, we have the following table:\\

\begin{center}
	\begin{tabular}{||c|ccccc}
		
		place $v$ & $\infty$& 2&3 &7&	113\\
		\hline
		$P_v$ &id&$\{(-2 \cdot 113, 0), (-113, 0)\}$&id&$\{(-2 \cdot 113, 0), (-113, 0)\}$&id\\
		$\delta_2(P_v)$&id&$(1, 2, -2, -2, 2)$&id&(1, 1, 7, 7, 1)&id\\
		
		$a_{1, v}$&id&$(1, 1, -1, 1, -1)$&id&(1, 1, 7, 1, 7)&id\\
		
		$\delta_2(P_v)-a_{1, v}$&id&$(1, 2, 2,-2, -2)$&id&(1, 1, 1, 7,7)&id\\
		
		$\rho_v$&id&$(-1, -2, 2)$&id&(7, 7, 1)&id\\

	\end{tabular}
\end{center}

$\;$\\

Following the explicit algorithm for computing the Cassels-Tate pairing, we get that the Cassels-Tate pairing between $(113, 113, 1)$ and $(2, 2, 1)$ is the only nontrivial one.  \\

Therefore, we get the $5\times 5$ Cassels-Tate pairing matrix from the 5 generators of $\text{Sel}^{\widehat{\phi}}(\widehat{J})$.  More specifically, the $ij^{th}$ entry of the matrix is the Cassels-Tate pairing between the $i^{th}$ and the $j^{th}$ generators of $\text{Sel}^{\widehat{\phi}}(\widehat{J})$, where the generators are in the same order as listed in the Selmer group $\text{Sel}^{\widehat{\phi}}(\widehat{J})$ \eqref{eqn: sel}.
\\
$$
\begin{bmatrix}
1&1&1&1&1\\
1&1&1&1&1\\
1&1&1&-1&1\\
1&1&-1&1&1\\
1&1&1&1&1\\

\end{bmatrix}
$$
$\;$\\

\begin{remark}

From the computation above, we have shown that the kernel of the Cassels-Tate pairing has dimension 3. We make the following observations:\\
\begin{itemize}
	\item We bound the rank of  $J(\Q)$ via bounding $|J(\Q)/\widehat{\phi}(\widehat{J}(\Q))|$ by $|\ker\langle \;,\;\rangle_{CT}|= 2^3$ instead of $|\text{Sel}^{\widehat{\phi}}(\widehat{J})| =2^5$. This improves the rank bound of $J(\Q)$ from $4$ to $2$.\\
	\item Consider the exact sequence \eqref{eqn: exact}. It can be shown that  $\Ima \alpha$ is contained inside $\ker \langle \;, \; \rangle_{CT}$, the kernel of the Cassels-Tate pairing on $\text{Sel}^{\widehat{\phi}}(\widehat{J}) \times \text{Sel}^{\widehat{\phi}}(\widehat{J})$. Indeed, if $a \in \text{Sel}^{\widehat{\phi}}(\widehat{J})$ is equal to $\alpha(b), $ where $b \in \text{Sel}^2(J)$, then following the earlier notations, we can let $a_1 = b$. Then we can pick $P_v \in J(\Q_v)$ to be the lift of $a_{1, v}$. Therefore,  $\delta_2(P_v)-a_{1,v}= 0 \in H^1(G_{\Q_v}, J[2])$ which implies, $a \in \ker\langle \;,\; \rangle_{CT}.$ Hence, we can always bound $|\text{Sel}^2(J)|$ and this bound will be sharp when $\Ima \alpha = \ker \langle \;,\;\rangle_{CT}$. \\
	
	We used MAGMA to compte the size of $\text{Sel}^{2}(J)$, which is equal to $2^6$, and we have the exact sequence:
	$$0 \rightarrow J[\phi](\Q) \rightarrow J[2](\Q) \rightarrow  \widehat{J}[\widehat{\phi}](\Q) \rightarrow \text{Sel}^{\phi}(J) \rightarrow \text{Sel}^2(J) \xrightarrow{\alpha}\ker \langle \;, \;\rangle_{CT}\rightarrow0.$$
	$$\text{size}=2^2 \;\;\;\ \text{size}=2^4   \;\;\;\ \text{size}=2^2   \;\;\;\ \text{size}=2^3    \;\;\;\ \text{size}=\boldsymbol{2^6}  \;\;\;\ \text{size}=2^3$$
	So for this example, we have turned the descent by Richelot isogeny into a 2-descent via computing the Cassels-Tate pairing. \\
	
\end{itemize}

\end{remark}

\end{document}